\documentclass[11pt]{article}
\usepackage{amsmath}
\usepackage{fancyvrb}
\usepackage[all]{xy}
\usepackage{tabularx}
\usepackage{amscd}
\usepackage{mhchem}
\usepackage{amsthm}
\usepackage{amsfonts}

\usepackage{hyperref}

\def\Hom{\mathop{\rm Hom}\nolimits}
\def\Ext{\mathop{\rm Ext}\nolimits}

\def\id{\mathop{\rm id}\nolimits}
\def\mod{\mathop{\rm mod}\nolimits}
\def\Mod{\mathop{\rm Mod}\nolimits}

\def\Add{\mathop{\rm Add}\nolimits}

\def\rep{\mathop{\rm rep}\nolimits}
\def\rank{\mathop{\rm rank}\nolimits}

\def\End{\mathop{\rm End}\nolimits}
\def\tilt{\mathop{\rm tilt}\nolimits}
\def\Proj{\mathop{\rm Proj}\nolimits}
\def\proj{\mathop{\rm proj}\nolimits}
\def\Gen{\mathop{\rm Gen}\nolimits}
\def\id{\mathop{\rm id}\nolimits}

\usepackage{ulem}
\begin{document}

\newcommand{\nc}{\newcommand}
\def\PP#1#2#3{{\mathrm{Pres}}^{#1}_{#2}{#3}\setcounter{equation}{0}}
\def\ns{$n$-star}\setcounter{equation}{0}
\def\nt{$n$-tilting}\setcounter{equation}{0}
\def\Ht#1#2#3{{{\mathrm{Hom}}_{#1}({#2},{#3})}\setcounter{equation}{0}}
\def\qp#1{{${(#1)}$-quasi-projective}\setcounter{equation}{0}}
\def\mr#1{{{\mathrm{#1}}}\setcounter{equation}{0}}
\def\mc#1{{{\mathcal{#1}}}\setcounter{equation}{0}}
\def\HD{\mr{Hom}_{\mc{D}}}
\def\HC{\mr{Hom}_{\mc{C}}}
\def\AdT{\mr{Add}_{\mc{T}}}
\def\adT{\mr{add}_{\mc{T}}}
\def\Kb{\mc{K}^b(\mr{Proj}R)}
\def\kb{\mc{K}^b(\mc{P}_R)}
\def\AdpC{\mr{Adp}_{\mc{C}}}


\newtheorem{theorem}{Theorem}[section]
\newtheorem{proposition}[theorem]{Proposition}
\newtheorem{lemma}[theorem]{Lemma}
\newtheorem{corollary}[theorem]{Corollary}
\newtheorem{conjecture}[theorem]{Conjecture}
\newtheorem{question}[theorem]{Question}
\newtheorem{definition}[theorem]{Definition}
\newtheorem{example}[theorem]{Example}

\newtheorem{remark}[theorem]{Remark}
\def\Pf#1{{\noindent\bf Proof}.\setcounter{equation}{0}}
\def\>#1{{ $\Rightarrow$ }\setcounter{equation}{0}}
\def\<>#1{{ $\Leftrightarrow$ }\setcounter{equation}{0}}
\def\bskip#1{{ \vskip 20pt }\setcounter{equation}{0}}
\def\sskip#1{{ \vskip 5pt }\setcounter{equation}{0}}
\def\bg#1{\begin{#1}\setcounter{equation}{0}}
\def\ed#1{\end{#1}\setcounter{equation}{0}}
\def\KET{T^{^F\bot}\setcounter{equation}{0}}
\def\KEC{C^{\bot}\setcounter{equation}{0}}

\renewcommand{\thefootnote}{\fnsymbol{footnote}}
\setcounter{footnote}{0}
%
%


\title{\bf Silting Modules over Triangular Matrix Rings
\thanks{This work was partially supported by NSFC (Grant Nos. 11971225, 11571164). } }
\footnotetext{
E-mail:~hpgao07@163.com, huangzy@nju.edu.cn}
\smallskip
\author{\small Hanpeng Gao, Zhaoyong Huang\thanks{Corresponding author.}\\
{\it \footnotesize Department of Mathematics, Nanjing University, Nanjing 210093, Jiangsu Province, P.R. China}}
\date{}
\maketitle
\baselineskip 15pt
%
%
\begin{abstract}
Let $\Lambda,\Gamma$ be rings and $R=\left(\begin{array}{cc}\Lambda & 0 \\ M & \Gamma\end{array}\right)$ the
triangular matrix ring with $M$ a $(\Gamma,\Lambda)$-bimodule. Let $X$ be a right $\Lambda$-module and
$Y$ a right $\Gamma$-module. We prove that $(X, 0)$$\oplus$$(Y\otimes_\Gamma M, Y)$ is a silting
right $R$-module if and only if both $X_{\Lambda}$ and $Y_{\Gamma}$ are silting modules and $Y\otimes_\Gamma M$
is generated by $X$. Furthermore, we prove that if $\Lambda$ and $\Gamma$ are finite dimensional algebras
over an algebraically closed field and $X_{\Lambda}$ and $Y_{\Gamma}$ are finitely generated, then
$(X, 0)$$\oplus$$(Y\otimes_\Gamma M, Y)$ is a support $\tau$-tilting $R$-module if and only if
both $X_{\Lambda}$ and $Y_{\Gamma}$ are support $\tau$-tilting modules, $\Hom_\Lambda(Y\otimes_\Gamma M,\tau X)=0$
and $\Hom_\Lambda(e\Lambda, Y\otimes_\Gamma M)=0$ with $e$ the maximal idempotent such that
$\Hom_\Lambda(e\Lambda, X)=0$.
\vspace{10pt}




\end{abstract}
%
\vskip 30pt

\section{Introduction}

Tilting modules are fundamental in the representation theory of algebras. It is important to construct
a new tilting module from a given one and mutation of tilting modules is a very effective way to do it.
Happel and Unger \cite{HU} gave some necessary and sufficient conditions under which
mutation of tilting modules is possible; however, mutation of tilting modules may not be realized.

As a generalization of tilting modules,  support $\tau$-tilting modules over
finite dimensional algebras were introduced by Adachi, Iyama and Reiten \cite{AIR}, and they showed that
mutation of all support $\tau$-tilting modules is possible. A new (support $\tau$-)tilting module can be
constructed by an algebra extension. For example, Assem, Happel and Trepode \cite{AHT} studied how to
extend and restrict tilting modules by given tilting modules for the one-point extension of an algebra by
a projective module. Suarez \cite{S} generalized this result to the case for support $\tau$-tilting modules.

To generalize tilting modules over an arbitrary ring and support $\tau$-tilting modules over a finite
dimensional algebra, (partial) silting modules over an arbitrary ring were introduced by Angeleri H\"ugel,
Marks and Vit\'oria \cite{AMV}. It was proved in \cite[Proposition 3.15]{AMV} that a finitely generated module is partial
silting (resp. silting) if and only if it is $\tau$-rigid (resp. support $\tau$-tilting) over a finite
dimensional algebra. Silting modules share many properties with tilting modules and support
$\tau$-tilting modules, see \cite{AMV1,AH,BF,BZ} and the references therein.

Let $\Lambda$, $\Gamma$ be rings and $M$ a $(\Gamma,\Lambda)$-bimodule. Then we can construct
the triangular matrix ring $\left(\begin{array}{cc}\Lambda & 0 \\
M & \Gamma\end{array}\right)$ by the ordinary operation on matrices, see \cite[p.76]{ARS}.
If $M$ and $N$ are two $\Lambda$-modules with $\Hom_\Lambda(N,M)=0$, then the endomorphism ring of
$M\oplus N$ is exactly the triangular matrix ring $\left(\begin{array}{cc}\End_\Lambda(M) & 0 \\
\Hom_\Lambda(M,N) & \End_\Lambda(N)\end{array}\right)$. Moreover, a one-point extension of an algebra
is a special triangular matrix algebra. In \cite{CGR}, Chen, Gong and Rump gave a criterion for lifting
tilting modules from an arbitrary ring to its trivial extension ring, and they constructed tilting modules
over triangular matrix rings under some conditions.
The aim of this paper is to construct (partial) silting modules over triangular matrix rings.
This paper is organized as follows.

In Section 2, we give some terminology and some preliminary results.
Let $\Lambda,\Gamma$ be rings and $R=\left(\begin{array}{cc}\Lambda & 0 \\
M & \Gamma\end{array}\right)$ the triangular matrix ring with $M$ a $(\Gamma,\Lambda)$-bimodule.
In Section 3, for any $X_{\Lambda}$ and $Y_{\Gamma}$, we investigate the relationship
between the projective presentations of $X_{\Lambda}$ and $Y_{\Gamma}$
and the projective presentation of the right $R$-module $(X, 0)$$\oplus$$(Y\otimes_\Gamma M, Y)$.
Then we give a necessary and sufficient condition for constructing (partial) silting right $R$-modules
from (partial) silting right $\Lambda$-modules and right $\Gamma$-modules (Theorem \ref{3.4}).
As a consequence, we get that if $_{\Gamma}M$ is flat, then
$(X,0)$$\oplus$$(Y\otimes_\Gamma M, Y)$ is a tilting right $R$-module if and only if
both $X_{\Lambda}$ and $Y_{\Gamma}$ are tilting and $Y\otimes_\Gamma M$ is generated by $X$ (Theorem \ref{3.8}).
In Section 4, $\Lambda$ and $\Gamma$ are finite dimensional $k$-algebras over an algebraically closed field $k$
and all modules considered are finitely generated and basic. As an application of Theorem \ref{3.4},
we give a necessary and sufficient condition for constructing support $\tau$-tilting right $R$-modules
from support $\tau$-tilting right $\Lambda$-modules and right $\Gamma$-modules (Theorem \ref{4.3}).
Furthermore, we generalize this result to tensor algebras (Theorem \ref{4.8}).
In Section 5, we give an example to illustrate our results; in particular, we may construct many support $\tau$-tilting
modules over triangular matrix algebras.

\section{Preliminaries}

Throughout this paper, all rings are associative with identities and all modules are unitary.
For a ring $\Lambda$, $\Mod \Lambda$ is the category of right $\Lambda$-modules, $\mod \Lambda$ is
the category of finitely generated right $\Lambda$-modules, and all subcategories
of $\Mod \Lambda$ or $\mod \Lambda$ are full and closed under isomorphisms.
We use $\Proj\Lambda$ (resp. $\proj\Lambda$) to denote the subcategory of $\Mod \Lambda$
(resp. $\mod\Lambda$) consisting of (resp. finitely generated) projective modules. For a module $M\in\Mod \Lambda$,
$\Add M$ is the subcategory of $\Mod \Lambda$ consisting of direct summands of direct sums of copies of $M$ and $\Gen M$
is the subcategory of $\Mod \Lambda$ consisting of quotients of direct sums of copies of $M$.

\subsection{Triangular matrix rings}

Let $\Lambda$, $\Gamma$ be rings and $_\Gamma M_\Lambda$ a ($\Gamma,\Lambda)$-bimodule. Then the triangular matrix ring
$$R:=\left(\begin{array}{cc}\Lambda & 0 \\
M & \Gamma\end{array}\right)$$
can be defined by the ordinary operation on matrices.
Let $\mathcal{C}_R$ be the category whose objects are the triples ${(X, Y)}_{f}$ with $X\in\Mod \Lambda$, $Y\in\Mod \Gamma$
and $f\in\Hom_\Lambda(Y\otimes_\Gamma M,X)$ (sometimes, $f$ is omitted).
The morphisms from ${(X, Y)}_{f}$ to ${(X', Y')}_{f'}$ are pairs of ${(\alpha, \beta)}$
such that the following diagram
$$\xymatrix{Y\otimes_\Gamma M\ar[d]_{\beta\otimes M}\ar[rr]^f&&X\ar[d]^\alpha\\
Y'\otimes_\Gamma M\ar[rr]^{f'}&&X'\\}$$
commutes, where $\alpha\in\Hom_\Lambda(X,X')$ and $\beta\in\Hom_\Gamma(Y,Y')$.

It is well known that there exists an equivalence of categories between $\Mod R$ and $\mathcal{C}_R$ (\cite{G}).
Hence we can view an $R$-module as a triples ${(X, Y)}_{f}$ with $X\in\Mod \Lambda$ and $Y\in\Mod \Gamma$.
Moreover, a sequence
$$0 \to {(X_1, Y_1)}\stackrel{{{(\alpha_1, \beta_1)}}}{\longrightarrow}
{(X_2, Y_2)}\stackrel{{{(\alpha_2, \beta_2)}}}{\longrightarrow}{(X_3, Y_3)}\to 0$$
in $\Mod R$ is exact if and only if
$$0 \to X_1\stackrel{\alpha_1}{\longrightarrow} X_2\stackrel{\alpha_2}{\longrightarrow}X_3\to 0$$
is exact in $\Mod \Lambda$ and
$$0 \to Y_1\stackrel{\beta_1}{\longrightarrow} Y_2\stackrel{\beta_2}{\longrightarrow}Y_3\to 0$$
is exact in $\Mod \Gamma$. All indecomposable projective modules in $\Mod R$ are exactly of the forms
${(P_\Lambda,  0)}$ and ${(Q_\Gamma\otimes_\Gamma M,  Q_\Gamma)}_{\id}$,
where $P_\Lambda$ is an indecomposable projective $\Lambda$-module and $Q_\Gamma$ is an indecomposable
projective $\Gamma$-module.

\subsection{Silting modules}
Let $\Lambda$ be a ring and
$$\sigma: P_{1}\rightarrow P_0$$ a homomorphism in $\Mod \Lambda$ with
$P_{1}, P_{0}\in \Proj\Lambda$. We write
\begin{center}
$D_\sigma:=\{A\in \Mod \Lambda\mid\Hom_\Lambda(\sigma, A)$ is epic$\}$.
\end{center}
Recall that a subcategory $\mathcal{T}$ of $\Mod \Lambda$ is called a {\it torsion class}
if it is closed under images, direct sums and extensions (c.f. \cite[Chapter VI]{ASS}).

\begin{definition}\label{2.1} {\rm (\cite[Definition 3.7]{AMV})
Let $T\in \Mod \Lambda$.
\begin{enumerate}
\item[(1)] $T$ is called {\it partial silting} if there exists a  projective presentation $\sigma$ of $T$ such that
$D_\sigma$ is a torsion class and $T\in D_\sigma$.
\item[(2)] $T$ is called {\it silting} if  there exists a  projective presentation $\sigma$ of $T$ such that  $\Gen T=D_\sigma$.
\end{enumerate}
Sometimes, we also say that $T$ is a ({\it partial}) {\it silting} module with respect to $\sigma$.}
\end{definition}

By \cite[Lemma 3.6(1)]{AMV}, $D_\sigma$ is always closed under images and extensions. Hence,
$D_\sigma$ is a torsion class if and only if it is closed under direct sums. This is always true when $\sigma$
is a map in $\proj\Lambda$. Moreover, it is trivial that $T\in D_\sigma$ implies $\Gen T\subseteq D_\sigma$.

Given a subcategory $\mathcal{X}$ of $\Mod \Lambda$, recall that a {\it left $\mathcal{X}$-approximation} of a module
$M\in\Mod \Lambda$ is a homomorphism $\phi:M\rightarrow X$ with $X\in \mathcal{X}$ such that $\Hom_\Lambda(\phi,X')$
is epic for any $X'\in \mathcal{X}$. The following result establishes the relation between partial silting modules
and silting modules.

\begin{proposition}\label{2.2} {\rm (\cite[Proposition 3.11]{AMV})}
Let $T\in \Mod \Lambda$  with a projective presentation $\sigma$. Then $T$ is  a silting module with respect to
$\sigma$ if and only if  $T$ is a partial silting module with respect to $\sigma$ and there exists an exact sequence
$$\Lambda\stackrel{\phi}\longrightarrow T^0\rightarrow T^1\rightarrow 0$$
in $\Mod \Lambda$ with $T^0, T^1\in \Add T$ and $\phi$ a left $D_\sigma$-approximation.
\end{proposition}

\subsection{Support $\tau$-tilting modules}
In this subsection, $\Lambda$ is a finite dimensional $k$-algebra over an algebraically closed field $k$.
The Auslander-Reiten translation is denoted by $\tau$.
For a module $M\in\mod \Lambda$, $|M|$ is the number of pairwise non-isomorphic direct summands of $M$. All modules
considered are finitely generated and basic.

\begin{definition}\label{2.3} {\rm (\cite[Definition 0.1]{AIR})
Let $M\in\mod\Lambda$.
\begin{enumerate}
\item[(1)] $M$ is called {\it $\tau$-rigid} if $\Hom_\Lambda(M,\tau M)=0$.
\item[(2)] $M$ is called {\it $\tau$-tilting}  if it is $\tau$-rigid and
$|M|=|\Lambda|$.
\item[(3)] $M$ is called {\it support $\tau$-tilting} if it is a $\tau$-tilting $\Lambda/\Lambda e\Lambda$-module
for some idempotent $e$ of $\Lambda$.
\end{enumerate}}
\end{definition}


\begin{lemma}\label{2.4}
If $M$ is a $\tau$-rigid $\Lambda$-module and $\Hom_\Lambda(e\Lambda,M)=0$ for some idempotent $e$
of $\Lambda$, then $|M|+|e\Lambda|\leq|\Lambda|$.
\end{lemma}

\begin{proof}
Let $M$ be a $\tau$-rigid $\Lambda$-module and $\Hom_\Lambda(e\Lambda,M)=0$ for some idempotent $e$
of $\Lambda$. Then $M$ is a $\tau$-rigid $\Lambda/\Lambda e\Lambda$-module by \cite[Lemma 2.1]{AIR}.
So $|M|+|e\Lambda|\leq|\Lambda|$.
\end{proof}

Sometimes, it is convenient to view support $\tau$-tilting modules and $\tau$-rigid modules as
certain pairs of modules in $\mod \Lambda$.

\begin{definition} \label{2.5} {\rm (\cite[Definition 0.3]{AIR})}
Let $(M,P)$ be a pair in $\mod\Lambda$ with $P\in \proj\Lambda$.
\begin{enumerate}
\item[(1)] $(M, P)$ is called a {\it $\tau$-rigid pair} if $M$ is $\tau$-rigid and $\Hom_\Lambda(P,M)=0$.
\item[(2)] $(M, P)$ is called a {\it support $\tau$-tilting pair} if $M$ is $\tau$-rigid and $|M|+|P|=|\Lambda|$.
\end{enumerate}
\end{definition}

It was shown in \cite[Proposition 2.3]{AIR} that $(M,P)$ is a support $\tau$-tilting pair in
$\mod \Lambda$ if and only if $M$ is a $\tau$-tilting $\Lambda/\Lambda e\Lambda$-module with $e\Lambda\cong P$.
Recall that $M\in\mod \Lambda$ is called {\it sincere} if there does not exist
a non-zero idempotent $e$ of $\Lambda$ that annihilates $M$. Notice that all $\tau$-tilting modules are sincere,
so $M$ is a support $\tau$-tilting $\Lambda$-module if and only if $M$ is a $\tau$-rigid $\Lambda$-module
and $|M|+|e\Lambda|=|\Lambda|$, where $e$ is the maximal idempotent such that $\Hom_\Lambda(e\Lambda,M)=0$.

\begin{lemma}\label{2.6} {\rm (\cite[Proposition 2.4]{AIR})}
Let $X\in\mod\Lambda$ and
$$P_1 \stackrel{f_0}{\longrightarrow}P_0 {\rightarrow} X{\rightarrow} 0$$ a minimal projective presentation
of $X$ in $\mod\Lambda$. For any $Y\in\mod \Lambda$, $\Hom_\Lambda(f_0,Y)$ is epic if and only if
$\Hom_\Lambda(Y,\tau X)=0$.
\end{lemma}

Silting modules are intended to generalize support $\tau$-tilting modules. In particular, when restricting to
finitely generated modules over a finite dimensional $k$-algebra, they are equivalent.

\begin{proposition} \label{2.7} {\rm (\cite[Proposition 3.15]{AMV})}
Let $T\in\mod\Lambda$. Then we have
\begin{enumerate}
\item[(1)] $T$ is a partial silting $\Lambda$-module if and only if $T$ is a $\tau$-rigid $\Lambda$-module.
\item[(2)] $T$ is a silting $\Lambda$-module if and only if $T$ is a support $\tau$-tilting $\Lambda$-module.
\end{enumerate}
\end{proposition}

Let $T\in\mod \Lambda$ with a minimal projective presentation $\sigma$.
Then $T$ is support $\tau$-tilting if and only if $\Gen T$ consists of $\Lambda$-modules $M$
such that $\Hom_\Lambda(\sigma\oplus\sigma', M)$ is epic, where $\sigma'$ is the complex $(e\Lambda\rightarrow0)$
and $e$ is a suitable idempotent of $\Lambda$ (\cite[Theorem 2.5]{AMV}). In fact, it follows from
\cite[Theorem 4.9]{AMV} and \cite[Theorem 3.2]{AIR} that $(T, e\Lambda)$ is a support
$\tau$-tilting pair if and only if $T$ is a silting module with respect to $\sigma\oplus\sigma'$.

\section{Silting modules over triangular matrix rings}

From now on, $\Lambda$, $\Gamma$ are rings and $_\Gamma M_\Lambda$ a $(\Gamma,\Lambda)$-bimodule
and $$R:=\left(\begin{array}{cc}\Lambda & 0 \\
M & \Gamma\end{array}\right)$$ is the corresponding triangular matrix ring.

Let $X\in\Mod \Lambda$ and $Y\in\Mod\Gamma$, and let
$$P_{1} \stackrel{\sigma_X}{\longrightarrow}P_0 {\rightarrow} X{\rightarrow} 0$$
and $$Q_{1}\stackrel{\sigma_Y}{\longrightarrow}Q_0 {\rightarrow} Y{\rightarrow} 0\eqno{(3.1)}$$
be projective presentations of $X$ and $Y$ respectively,
with $P_{1}, P_0\in \Proj\Lambda$ and $Q_{1}, Q_0\in \Proj\Gamma$.
Applying the functor $-\otimes_\Gamma M$ to (3.1), we get the following exact sequence
$$Q_1\otimes_\Gamma M \stackrel{\sigma_Y\otimes M}{\longrightarrow}
Q_0\otimes_\Gamma M {\rightarrow} Y\otimes_\Gamma M{\rightarrow} 0.$$
Hence, we get a projective presentation of $(X, 0)$$\oplus$$(Y\otimes_\Gamma M, Y)$
denoted by $\sigma=(\smallmatrix a &0\\0& b\endsmallmatrix)$:
\begin{center}
$(P_{1}, 0)$$\oplus$$(Q_{1}\otimes_\Gamma M,  Q_{1})$
$\stackrel{\sigma}{\longrightarrow}$$(P_0,0)$$\oplus$$(Q_0\otimes_\Gamma M, Q_0)$
${\rightarrow}$$(X, 0)$$\oplus$$(Y\otimes_\Gamma M, Y)$${\rightarrow} 0,$
\end{center}
where $a=$$({\sigma_X}, 0)$ and $b=$$({\sigma_Y}\otimes M, {\sigma_Y})$.

\begin{lemma}\label{3.1}
Let $X_1\in\Mod \Lambda$ and $Y_1\in\Mod\Gamma$.
\begin{enumerate}
\item[(1)] ${({X_1}, {Y_1})}_h\in D_{\sigma}$ if and only if  $X_1\in D_{\sigma_X}$
and $Y_1\in D_{\sigma_Y}$.
\item[(2)] If $X_1\in D_{\sigma_X}$, then $({X_1}, 0)\in D_{\sigma}$.
\item[(3)] If $Y_1\in D_{\sigma_Y}$, then $(0, {Y_1})\in D_{\sigma}$.
\item[(4)] If $Y_1\in D_{\sigma_Y}$ and $Y_1\otimes_\Gamma M\in D_{\sigma_X}$,
then $({Y_1\otimes_\Gamma M, Y_1)}_{\id}\in D_{\sigma}$.
\end{enumerate}
\end{lemma}

\begin{proof}
(1) Let $f\in\Hom_{\Lambda}(P_1,X_1)$. Then
$((f, 0), 0)\in\Hom_R((P_1,0)\oplus(Q_1\otimes_\Gamma M, Q_1),(X_1, Y_1))$.
Since $(X_1, Y_1)_h\in D_{\sigma}$,
there exists $(({f'}, 0),y):(P_0,0)\oplus(Q_0\otimes_\Gamma M, Q_0)\to(X_1,Y_1)$
with $f'\in\Hom_\Lambda(P_0,X_1)$ such that the following diagram
$$\xymatrix@C=15pt{{(P_1,0)} \oplus{(Q_1\otimes_\Gamma M,Q_1)}
\ar[rr]^{\sigma}\ar[d]_{({(f,0)},0)}&&{(P_0,0)}
\oplus{(Q_0\otimes_\Gamma M,Q_0)}\ar@{.>}[lld]^{({(f',0)},y)}\\
{(X_1, Y_1)}&&
}$$ commutes. So $((f,0),0)=(({f'},0),y)\circ\sigma$, and
hence ${(f,0)}={(f',0)}\circ a={(f',0)}\circ{(\sigma_X,0)}$ and
$f=f'\circ \sigma_X$. It implies $X_1\in D_{\sigma_X}$.

Let $g\in\Hom_\Gamma(Q_1,Y_1)$. Then $(0,(h\circ(g\otimes M),g))$
$\in\Hom_R((P_1,0)\oplus(Q_1\otimes_\Gamma M,Q_1)$,
$(X_1,Y_1)$) and there exists $(x,{({f_1},  g_1)})\in\Hom_R((P_0,0)
\oplus(Q_0\otimes_\Gamma M, Q_0)$,
${(X_1, Y_1)})$ with $f_1\in\Hom_\Lambda(Q_0\otimes_\Gamma M,X_1)$
and $g_1\in\Hom_\Gamma(Q_0,Y_1)$ such that the following diagram
$$\xymatrix@C=15pt{({P_1,0}) \oplus{(Q_1\otimes_\Gamma M,Q_1)}
\ar[rr]^{\sigma}\ar[d]_{(0,({h\circ(g\otimes M),g}))}&&{(P_0,0)}
\oplus{(Q_0\otimes_\Gamma M,Q_0)}\ar@{.>}[lld]^{(x,{({f_1},g_1}))}\\
{(X_1,Y_1)}_f&&}$$ commutes. So
$(0, ({h\circ (g\otimes M), g}))=(x,{({f_1},g_1}))\circ\sigma$, and hence
$({h\circ(g\otimes M), g})={({f_1},g_1})\circ b={({f_1},g_1})\circ{({\sigma_Y}
\otimes_\Gamma M,{\sigma_Y})}$
and $g=g_1\circ \sigma_Y$. It implies $Y_1\in D_{\sigma_Y}$.

Conversely, let $(x,y)\in\Hom_R({(P_1,0)}\oplus{(Q_1\otimes_\Gamma M, Q_1)}$,
${(X_1, Y_1)}_{h}$). Write $x={(f_2, 0)}$ and $y={(f_3, g_3)}$
with $f_2\in\Hom_\Lambda(P_1, X_1)$, $f_3\in\Hom_\Lambda(Q_1\otimes_\Gamma M, X_1)$
and $g_3\in\Hom_\Gamma(Q_1, Y_1)$. Then we have the following commutative diagram
$$\xymatrix{Q_1\otimes_\Gamma M\ar[d]_{g_3\otimes M}\ar[rr]^{\id}&&Q_1\otimes_\Gamma M\ar[d]^{f_3}\\
Y_1\otimes_\Gamma M\ar[rr]^{h}&&X_1.}$$
and so $f_3=h\circ (g_3\otimes M)$. Since $X_1\in D_{\sigma_X}$ and $Y_1\in D_{\sigma_Y}$,
there exist $f'_2\in\Hom_\Lambda(P_0,X_1)$ and $g'_3\in\Hom_\Gamma(Q_0, Y_1)$ such that
$f_2=f'_2\circ \sigma_X$ and $g_3=g'_3\circ \sigma_Y$.
Since ${(h\circ (g'_3\otimes M), g'_3)}\in\Hom_{R}({(Q_0\otimes_\Gamma M, Q_0)},{(X_1, Y_1)})$
and the following equalities hold
\begin{equation*}
\begin{split}
\left({(f'_2, 0)}, {(h\circ (g'_3\otimes M), g'_3)}\right)\circ\sigma
&=\left({(f'_2, 0)}, {(h\circ (g'_3\otimes M), g'_3)}\right)\circ\left(\smallmatrix a &0\\0& b\endsmallmatrix\right)\\
&=\left({(f'_2, 0)}\circ a, {(h\circ (g'_3\otimes M), g'_3)}\circ b\right)\\
&=\left({(f'_2, 0)}\circ {(\sigma_X, 0)}, {(h\circ (g'_3\otimes M), g'_3)}\circ{(\sigma_Y\otimes M, \sigma_Y)}\right)\\
&=\left({({f'_2\circ\sigma_X}, 0)}, {\left({h\circ ((g'_3\circ\sigma_Y}\right)\otimes M), {g'_3\circ\sigma_Y}) }\right)\\
&=\left({({f_2}, 0)}, {(h\circ ({g_3}\otimes M), {g_3}) }\right)\\
&=\left({({f_2}, 0)}, {{(f_3}, {g_3}) }\right)\\
&=(x,y),
\end{split}
\end{equation*}
we have ${(X_1, Y_1)}_{h}$$\in D_{\sigma}$.

The assertions (2), (3) and  (4) follow directly from (1).
\end{proof}

Let $I$ be a set and $\{{(X_i, Y_i)}_{f_i} \}_{i\in I}\in \Mod R$ with
all $X_i\in\Mod \Lambda$ and $Y_i\in\Mod \Gamma$. Since the tensor functor commutes
with direct sums, we have
$\bigoplus_{i\in I}{(X_i, Y_i)}_{f_i}\cong$
$( \bigoplus_{i\in I}X_i ,\bigoplus_{i\in I}Y_i)_{{\bigoplus_{i\in I}f_i}}$.

\begin{lemma}\label{3.2}
$D_\sigma$ is a torsion class if and only if both $D_{\sigma_X}$ and $D_{\sigma_Y}$ are torsion classes.
\end{lemma}

\begin{proof}
Suppose that $D_{\sigma}$ is a torsion class.
Let $\{X_i\}_{i\in I}$ be a family of modules in $D_{\sigma_X}$. Then ${(X_i, 0)}\in D_{\sigma}$
for any $i\in I$ by Lemma 3.1(2). So ${({\bigoplus_{i\in I}X_i}, 0)}\cong\bigoplus_{i\in I}{(X_i, 0)}\in D_{\sigma}$,
and hence $\bigoplus_{i\in I}X_i\in D_{\sigma_X}$ by Lemma 3.1(1). Thus $D_{\sigma_{X}}$ is a torsion class
by \cite[Lemma 3.6(1)]{AMV}. Similarly,  $D_{\sigma_{Y}}$ is also a torsion class.

Conversely, suppose that both $D_{\sigma_X}$ and $D_{\sigma_Y}$ are torsion classes. Let
$\{{(X_i, Y_i)}_{f_i} \}_{i\in I}$ be a family of modules in $D_{\sigma}$ with $X_i\in\Mod \Lambda$ and $Y_i\in\Mod \Gamma$.
Then $X_i\in D_{\sigma_{X}}$ and $Y_i\in D_{\sigma_{Y}}$ for any $i\in I$ by Lemma 3.1(1).
Hence $\bigoplus_{i\in I}X_i\in D_{\sigma_{X}}$ and
$\bigoplus_{i\in I}Y_i\in D_{\sigma_{Y}}$. Let $(x,y)\in\Hom_R$$({(P_1, 0)}\oplus{(Q_1\otimes_\Gamma M, Q_1)},$
$( \bigoplus_{i\in I}X_i,\bigoplus_{i\in I}Y_i)_{{\bigoplus_{i\in I}f_i}}$).
Write $x={(f_1, 0)}$ and $y={(f_2, g_2)}$ with $f_1\in\Hom_\Lambda(P_1, \bigoplus_{i\in I}X_i )$,
$f_2\in\Hom_\Lambda(Q_1\otimes_\Gamma M, \bigoplus_{i\in I}X_i)$ and $g_2\in\Hom_\Gamma(Q_1, \bigoplus_{i\in I}Y_i)$.
Then we have the following commutative diagram
$$\xymatrix{Q_1\otimes_\Gamma M\ar[d]_{g_2\otimes M}\ar[rr]^{\id}&&Q_1\otimes_\Gamma M\ar[d]^{f_2}\\
(\bigoplus_{i\in I}Y_i)\otimes_\Gamma M\ar[rr]^{\bigoplus_{i\in I}f_i}&&\bigoplus_{i\in I}X_i,\\}$$
and so $f_2=(\bigoplus_{i\in I}f_i)\circ(g_2\otimes M)$. Since $\bigoplus_{i\in I}X_i\in D_{\sigma_{X}}$
and $\bigoplus_{i\in I}Y_i\in D_{\sigma_{Y}}$, there exist
$f'_1\in\Hom_\Lambda(P_0,\bigoplus_{i\in I}X_i)$ and $g'_2\in\Hom_\Gamma(Q_0, \bigoplus_{i\in I}Y_i)$
such that $f_1=f'_1\circ\sigma_X$  and $g_2=g'_2\circ\sigma_Y$.

Set $g'_1:=(\bigoplus_{i\in I}f_i)\circ(g'_2\otimes M)$. Then $(g'_1, g'_2)$$\in\Hom_R$(${(Q_0\otimes_\Gamma M, Q_0)},$
$( \bigoplus_{i\in I}X_i,\bigoplus_{i\in I}Y_i))$.
Since the following equalities hold
\begin{equation*}
\begin{split}
\left({(f'_1, 0)},{(g'_1, g'_2)}\right)\circ\sigma
&=\left({(f'_1, 0)},{(g'_1, g'_2)}\right)\circ(\smallmatrix a &0\\0& b\endsmallmatrix)\\
&=\left({(f'_1,0)}\circ a,{(g'_1,g'_2)}\circ b\right)\\
&=\left({(f'_1, 0)}\circ {(\sigma_X, 0)},{(g'_1, g'_2)}\circ{(\sigma_Y\otimes M, \sigma_Y)}\right)\\
&=\left({{(f'_1\circ\sigma_X},0)},{{(g'_1\circ(\sigma_Y\otimes M)}, {g'_2\circ\sigma_Y}) }\right)\\
&=\left({{(f_1}, 0)},{{((\bigoplus_{i\in I}f_i)\circ(g'_2\otimes M)\circ(\sigma_Y\otimes M)}, {g_2} )}\right)\\
&=\left({{(f_1}, 0)},{{((\bigoplus_{i\in I}f_i)\circ(g'_2\circ\sigma_Y)\otimes M}, {g_2}) }\right)\\
&=\left({{(f_1}, 0)},{{((\bigoplus_{i\in I}f_i)\circ(g_2\otimes M)}, {g_2}) }\right)\\
&=\left({({f_1}, 0)},{{(f_2}, {g_2}) }\right)\\
&=(x,y),
\end{split}
\end{equation*}
we have $\bigoplus_{i\in I}{(X_i, Y_i)}_{f_i}\cong$
$( \bigoplus_{i\in I}X_i,\bigoplus_{i\in I}Y_i)_{{\bigoplus_{i\in I}f_i}}\in D_{\sigma}$.
Thus $D_{\sigma}$ is a torsion class by \cite[Lemma 3.6(1)]{AMV} again.
\end{proof}

As an immediate consequence of Lemmas \ref{3.1} and \ref{3.2}, we get the following result.

\begin{proposition}\label{3.3}
Let $X\in\Mod\Lambda$ with a projective presentation $\sigma_{X}$
and $Y\in\Mod\Gamma$ with a projective presentation $\sigma_{Y}$.
Then $(X, 0)$$\oplus$$(Y\otimes_\Gamma M,Y)$ is a partial silting $R$-module
with respect to $\sigma$ if and only if the following conditions are satisfied.
\begin{enumerate}
\item[(i)] $X$ is a partial silting $\Lambda$-module with respect to $\sigma_X$.
\item[(ii)] $Y$ is a partial silting $\Gamma$-module with respect to $\sigma_Y$.
\item[(iii)] $Y\otimes_\Gamma M\in D_{\sigma_{X}}$.
\end{enumerate}
\end{proposition}

The following is the main result in this section.

\begin{theorem}\label{3.4}
Let $X\in\Mod\Lambda$ and $Y\in\Mod\Gamma$. Then $(X, 0)$$\oplus$$(Y\otimes_\Gamma M,Y)$ is a silting
$R$-module if and only if the following conditions are satisfied.
\begin{enumerate}
\item[(i)] $X$ is a silting $\Lambda$-module.
\item[(ii)] $Y$ is a silting $\Gamma$-module.
\item[(iii)] $Y\otimes_\Gamma M\in \Gen X$.
\end{enumerate}
\end{theorem}

\begin{proof}
Let $(X, 0)$$\oplus$$(Y\otimes_\Gamma M, Y)$ be a silting $R$-module with respect to $\sigma$. Then
$$D_{\sigma}=\Gen\left({(X, 0)}\oplus{(Y\otimes_\Gamma M, Y)}\right)=
\Gen{(X, 0)}\oplus\Gen{(Y\otimes_\Gamma M, Y)}.$$
Since $X$ is partial silting by Proposition \ref{3.3}, we have $\Gen X\subseteq D_{\sigma_X}$.
Let $X_1\in  D_{\sigma_X}$. Then ${(X_1,0)}\in D_{\sigma}$. If ${(X_1, 0)}$ has a direct summand
${(X'_1, 0)}$ in $\Gen{(Y\otimes_\Gamma M, Y)}$, then we have the following commutative diagram
with exact columns
$$\xymatrix{(Y\otimes_\Gamma M)^{(I)}\ar[d]_{0}\ar[rr]^{\id}&&(Y\otimes_\Gamma M)^{(I)}\ar[d]^\alpha\\
0\ar[rr]\ar[d]&&X'_1\ar[d]\\
0&&0,\\}$$
and so $X'_1=0$. Thus ${(X_1, 0)}\in \Gen{(X, 0)}$ and $X_1\in \Gen X$. It follows that
$\Gen X=D_{\sigma_X}$ and $X$ is a silting $\Lambda$-module. By Proposition \ref{3.3} again,
we have $Y\otimes_\Gamma M\in D_{\sigma_{X}}=\Gen X$. Similarly, we have that $Y$ is a silting $\Gamma$-module.

Conversely, let $X$ be a silting $\Lambda$-module with respect to $\sigma_X$ and $Y$  a silting $\Gamma$-module
with respect to $\sigma_Y$ such that $Y\otimes_\Gamma M\in \Gen X$. Then
$(X, 0)$$\oplus$$(Y\otimes_\Gamma M, Y)$ is a partial silting module
with respect to $\sigma$ by Proposition \ref{3.3}. According to Proposition 2.2, we have the following exact sequences
$$\Lambda \stackrel{\phi}{\longrightarrow}T^1 {\rightarrow} T^2{\rightarrow} 0$$
and
$$\Gamma \stackrel{\psi}{\longrightarrow}E^1 {\rightarrow} E^2{\rightarrow} 0$$
with $T^{1}, T^2\in \Add X$ and $E^{1}, E^2\in \Add Y$ such that $\phi$ is a left $D_{\sigma_X}$-approximation
and $\psi$ is a left $D_{\sigma_Y}$-approximation. Set
$$a':={({\phi}, 0)},\  b':={({\psi}\otimes_\Gamma M, {\psi})}\ \text{and}\
\alpha:=\left(\begin{array}{cc}a' & 0 \\ 0 & b'\end{array}\right).$$
Then we get the following exact sequence
$${(\Lambda, 0)}\oplus{(M,\Gamma)} \stackrel{\alpha}{\longrightarrow}
{(T^1,0)}\oplus{(E^1\otimes_\Gamma M, E^1)}{\rightarrow}
{(T^2, 0)}\oplus{(E^2\otimes_\Gamma M,E^2)}{\rightarrow} 0.$$
Clearly, both $(T^1, 0)$$\oplus$$(E^1\otimes_\Gamma M, E^1)$ and
$(T^2, 0)$$\oplus$$(E^2\otimes_\Gamma M, E^2)$ belong to
$\Add({(X, 0)}$$\oplus$${(Y\otimes_\Gamma M,Y)})$.

Let ${(X_1, Y_1)}_{h}\in D_{\sigma}$
and $(x,y)\in\Hom_R({(\Lambda, 0)}\oplus{(M, \Gamma)}, {(X_1, Y_1)})$.
By Lemma \ref{3.1}(1), we have $X_1\in D_{\sigma_X}$ and $Y_1\in D_{\sigma_Y}$.
Write $x={(f_1, 0)}$ and $y={(f_2,  g_2)}$, we have $f_2=h\circ(g_2\otimes M)$.
Since $\phi$ is a left $D_{\sigma_X}$-approximation and $\psi$ is a left $D_{\sigma_Y}$-approximation,
there exist $f'_1\in\Hom_\Lambda(T^1, X_1)$ and $g'_2\in\Hom_\Gamma(E^1, Y_1)$ such that $f_1=f'_1\circ\phi$
and $g_2=g'_2\circ\psi$.
Take $f_3:=h\circ(g'_2\otimes M)(\in\Hom_{\Lambda}(E^1\otimes_\Gamma M, X_1))$.
Since the following equalities hold
\begin{equation*}
\begin{split}
\left({(f'_1, 0)}, {(f_3, g'_2)}\right)\circ\alpha
&=\left({(f'_1, 0)}, {(f_3, g'_2)}\right)\circ(\smallmatrix a' &0\\0& b'\endsmallmatrix)\\
&=\left({(f'_1, 0)}\circ a', {(f_3, g'_2)}\circ b'\right)\\
&=\left({(f'_1, 0)}\circ {(\phi, 0)}, {(f_3, g'_2)}\circ{(\psi\otimes M, \psi)}\right)\\
&=\left({{(f'_1\circ\phi}, 0)}, {{(f_3\circ(\psi\otimes M)}, {g'_2\circ\psi} )}\right)\\
&=\left({({f_1}, 0)}, {{(h\circ(g'_2\otimes M)\circ(\psi\otimes M)}, {g_2} )}\right)\\
&=\left({{(f_1}, 0)}, {{(h\circ(g'_2\circ\psi)\otimes M}, {g_2} )}\right)\\
&=\left({{(f_1}, 0)}, {{(h\circ(g_2\otimes M)}, {g_2}) }\right)\\
&=\left({{(f_1}, 0)}, {({f_2},  {g_2}) }\right)\\
&=(x, y),
\end{split}
\end{equation*}
that is, the following diagram
$$\xymatrix@C=15pt{{(\Lambda, 0)} \oplus{(M, \Gamma)}\ar[rr]^{\alpha~~~~~~}\ar[d]_{(x,~y)}&
&{(T^{1},0)} \oplus{(E^{1}\otimes_\Gamma M, E^{1})}\ar@{.>}[lld]^{({(f'_1,~ 0)},~{(f_3, ~g'_2)})}\\
{(X_1, Y_1)}&&}$$
commutes, we have that $\alpha$ is a left $D_{\sigma}$-approximation. It follows from Proposition \ref{2.2} that
$(X,0)$$\oplus$$(Y\otimes_\Gamma M,Y)$ is a silting $R$-module.
\end{proof}

\begin{remark}\label{3.5}
{\rm The sufficiency of the above theorem can be obtained by considering a special ring extension.
Take $S=\Lambda\times\Gamma$. Then we have a split surjective morphism $R\to S$ whose kernel is $_SM_S$.
Of course, we have $(X\oplus Y)\otimes R\cong $ $(X,0)$$\oplus$$(Y\otimes_\Gamma M, Y)$ and $R_S\cong S\oplus M$.
If $X$ is a silting $\Lambda$-module and  $Y$ is a silting $\Gamma$-module, then $X\oplus Y$ is a silting $S$-module.
It follows from \cite[Theorem 4.9]{AMV} that $\sigma_X$  and $\sigma_Y$ are 2-term silting complexes. Hence
$(X\oplus Y)\otimes R$ is a silting $R$-module if and only if $\sigma$ is a 2-term silting complex, and if and only if
$\Hom_R(R_S, (X\oplus Y)\otimes R)\in \Gen(X\oplus Y)$ by \cite[Theorem 2.2]{B}. Since
$$\Hom_R(R_S, (X\oplus Y)\otimes R)\cong  (X\oplus Y)\otimes R_S\cong (X\oplus Y)\oplus (X\oplus Y)\otimes_SM,$$
we have that $\Hom_R(R_S, (X\oplus Y)\otimes R)\in \Gen(X\oplus Y)$ if and only if $(X\oplus Y)\otimes_SM\in \Gen(X\oplus Y)$,
that is, $Y\otimes_\Gamma M\in \Gen X$ since $X\otimes_SM=0$.}
\end{remark}

\begin{corollary}\label{3.6}
Let $X\in\Mod\Lambda$ and $Y\in\Mod\Gamma$. Then
\begin{enumerate}
\item[(1)] $(X,0)$ is a silting $R$-module if and only if $X$ is a
silting $\Lambda$-module.
\item[(2)] ${(\Lambda,0)}\oplus{(Y\otimes_\Gamma M,Y})$ is a
silting $R$-module if and only if $Y$ is a silting $\Gamma$-module.
\end{enumerate}
\end{corollary}

\begin{proof}
(1) It is clear.

(2) Since $\Lambda$ is a silting $\Lambda$-module
 and  $\Gen\Lambda=\Mod \Lambda$,  the assertion follows immediately from Theorem \ref{3.4}.
\end{proof}

Recall from \cite{CT} that a module $X\in\Mod\Lambda$ is called {\it tilting} if $\Gen X=\{N\in\Mod \Lambda\mid\Ext_\Lambda^1(X,N)=0\}$;
or equivalently, if $X$ satisfies the following conditions.
\begin{enumerate}
\item[(i)] The projective dimension of $X$ is at most one.
\item[(ii)] $\Ext_\Lambda^1(X,X^{(I)})=0$ for any set $I$.
\item[(iii)] There exists an exact sequence
$$0\rightarrow \Lambda\rightarrow T_0\rightarrow T_1\rightarrow 0$$
in $\Mod\Lambda$ with $T^0,T^1\in\Add X$.
\end{enumerate}
It follows from \cite[Proposition 3.13(1)]{AMV} that a module $X\in\Mod\Lambda$ is tilting if and only if it is a silting
with respect to a monomorphic projective presentation.

\begin{corollary}\label{3.7}
Let $X\in\Mod\Lambda$ and $Y\in\Mod \Gamma$ be silting with monomorphic projective presentations $\sigma_{X}$ and
$\sigma_{Y}$ respectively. If $\sigma_Y\otimes_\Gamma M$ is monic and $Y\otimes_\Gamma M\in \Gen X$,
then  $(X, 0)$$\oplus$$(Y\otimes_\Gamma M, Y)$ is a tilting $R$-module.
\end{corollary}

\begin{proof}
It follows that $\sigma$ is monic when $\sigma_Y\otimes_\Gamma M$ is monic.
\end{proof}

Let $X_\Lambda$ and $Y_\Gamma$ be tilting modules, and let $\varepsilon^X$ and $\varepsilon^Y$ denote the
counits of the corresponding adjunctions, see \cite[p.535]{CGR}. It was proved in \cite[p.538, Corollary]{CGR}
that if $_\Gamma M$ is flat such that the functor $F=-\otimes_\Gamma M: \Mod \Gamma\rightarrow \Mod \Lambda$
satisfies $F\varepsilon^Y=\varepsilon^X F$, then $(X, 0)$$\oplus$$(Y\otimes_\Gamma M, Y)$ is a
tilting $R$-module.
The following theorem extends this result. The sufficiency of this theorem was obtained independently in
\cite[Theorem 5.2]{YM1986} by considering an epimorphism of rings which is split.

\begin{theorem}\label{3.8}
Let $X\in\Mod\Lambda$ and $Y\in\Mod\Gamma$. If $_\Gamma M$ is flat, then the following statements are equivalent.
\begin{enumerate}
\item[(1)] $(X, 0)$$\oplus$$(Y\otimes_\Gamma M, Y)$ is a tilting $R$-module.
\item[(2)] $X$ is a tilting $\Lambda$-module, $Y$ is a tilting $\Gamma$-module and $Y\otimes_\Gamma M\in \Gen X$.
\end{enumerate}
\end{theorem}

\begin{proof}
Since $\sigma$ is monic if and only if both  $\sigma_X$ and  $\sigma_Y$ are monic when $_\Gamma M$ is flat,
the assertion follows from Theorem \ref{3.4}.
\end{proof}

Applying Theorem \ref{3.4} and Theorem \ref{3.8} to the special triangular matrix ring
$\left(\begin{array}{cc}\Lambda & 0 \\ \Lambda & \Lambda\end{array}\right)$, we get the following result.

\begin{corollary}\label{3.9}
Let $X,Y\in\Mod \Lambda$ and $R=\left(\begin{array}{cc}\Lambda & 0 \\
\Lambda & \Lambda\end{array}\right)$. Then the following statements are equivalent.
\begin{enumerate}
\item[(1)] ${(X, 0)}\oplus{(Y, Y)}$ is a silting (resp. tilting) $R$-module.
\item[(2)] Both $X$ and $Y$ are silting (resp. tilting) $\Lambda$-modules and $Y\in \Gen X$.
\end{enumerate}
In particular, ${(X, 0)}\oplus{(X, X)}$ is a silting (resp. tilting) $R$-module
if and only if $X$ is a silting (resp. tilting) $\Lambda$-module.
\end{corollary}

\section{Support $\tau$-tilting modules over triangular matrix algebras}
In this section, all modules considered are finitely generated modules over finite dimensional $k$-algebras
over an algebraically closed field $k$.

\begin{proposition}\label{4.1}
Let $X\in\mod\Lambda$ and $Y\in\mod\Gamma$. Then $(X,0)$$\oplus$$(Y\otimes_\Gamma M,Y)$
is a $\tau$-rigid $R$-module if and only if the following conditions are satisfied.
\begin{enumerate}
\item[(1)] $X$ is a $\tau$-rigid $\Lambda$-module.
\item[(2)] $Y$ is a $\tau$-rigid $\Gamma$-module.
\item[(3)] $\Hom_\Lambda(Y\otimes_\Gamma M,\tau X)=0$.
\end{enumerate}
\end{proposition}

\begin{proof}
Considering the minimal projective presentation $\sigma_X$ of $X$, we have that $Y\otimes_\Gamma M\in D_{\sigma_X}$
if and only if $\Hom_\Lambda(Y\otimes_\Gamma M,\tau X)=0$ by Lemma 2.6. Thus the assertion follows from
Propositions \ref{2.7} and \ref{3.3}.
\end{proof}

For any $X\in\mod\Lambda$, it is clear that $(X, 0)$ is indecomposable if and only if so is $X$.
Let $Y\in\mod\Gamma$. Then ${(Y\otimes_\Gamma M, Y)}_{\id}$ is indecomposable implies so is $Y$;
conversely, assume that $Y$ is indecomposable and write ${(Y\otimes_\Gamma M,Y)}_{\id}
={(X_1, Y_1)}\oplus{(X_2, Y_2)}$ with $X_1,X_2\in \mod \Lambda$ and $Y_1,Y_2\in \mod \Gamma$.
Then either $Y_1=0$ or $Y_2=0$. If $Y_1=0$, then there exists a split epimorphism
$(\alpha,0):{(Y\otimes_\Gamma M, Y)}_{\id}\to {(X_1, 0)}$ in $\mod R$,
which implies $\alpha =0$. So $X_1=0$ and ${(X_1, Y_1)}=0$. Similarly, if $Y_2=0$, then ${(X_2, Y_2)}=0$.
Thus we conclude that ${(Y\otimes_\Gamma M,Y)}_{\id}$ is also indecomposable.
This proves the following lemma.

\begin{lemma}\label{4.2}
For any $X\in\mod\Lambda$ and $Y\in\mod\Gamma$, we have
$$|{(X, 0)}|+|{(Y\otimes_\Gamma M,Y)}|=|X|+|Y|.$$
\end{lemma}

The following is the main result in this section.

\begin{theorem}\label{4.3}
Let $X\in\mod\Lambda$ and $Y\in\mod\Gamma$. Then $(X, 0)$$\oplus$$(Y\otimes_\Gamma M, Y)$
is a support $\tau$-tilting $R$-module if and only if the following conditions are satisfied.
\begin{enumerate}
\item[(1)] $X$ is a support $\tau$-tilting $\Lambda$-module.
\item[(2)] $Y$ is  a support $\tau$-tilting $\Gamma$-module.
\item[(3)] $\Hom_\Lambda(Y\otimes_\Gamma M,\tau X)=0$.
\item[(4)] $\Hom_\Lambda(e\Lambda, Y\otimes_\Gamma M)=0$, where $e$ is the maximal idempotent
such that $\Hom_\Lambda(e\Lambda, X)=0$.
\end{enumerate}
\end{theorem}

\begin{proof}
Assume that $({(X,0)}\oplus{(Y\otimes_\Gamma M, Y)},
{(e\Lambda, 0)}\oplus{(e'\Gamma\otimes_\Gamma M, e'\Gamma)})$ is a support $\tau$-tilting pair
in $\mod R$, where $e$ and $e'$ are idempotents of $\Lambda$ and $\Gamma$ respectively.
Then $\Hom_\Lambda(Y\otimes_\Gamma M,\tau X)=0$ and both $X$ and $Y$ are $\tau$-rigid by Proposition \ref{4.1}.
Moreover, we have
$$\Hom_\Lambda(e\Lambda,X)=0,\Hom_\Gamma(e'\Gamma,Y)=0,\Hom_\Lambda(e\Lambda, Y\otimes_\Gamma M)=0.$$
By Lemma \ref{2.4}, we have $|X|+|e\Lambda|\leq|\Lambda|$ and $|Y|+|e'\Gamma|\leq|\Gamma|$. Note that
$$|X|+|Y|+|e\Lambda|+|e'\Gamma|=|{(X, 0)}|+|{(Y\otimes_\Gamma M, Y)}|+|{(e\Lambda, 0)}|
+|{(e'\Gamma\otimes_\Gamma M, e'\Gamma)}|=|R|=|\Lambda|+|\Gamma|$$
by Lemma \ref{4.2}. So $|X|+|e\Lambda|=|\Lambda|$ and $|Y|+|e'\Gamma|=|\Gamma|$,
and hence both $X$ and $Y$ are support $\tau$-tilting. Moreover, the
support $\tau$-tilting pair $(X, e\Lambda)$ implies that  $e$ is the maximal idempotent
such that $\Hom_\Lambda(e\Lambda, X)=0$.

Conversely, assume that the conditions (1)--(4) are satisfied. Then $(X, 0)$$\oplus$$(Y\otimes_\Gamma M, Y)$
is a $\tau$-rigid $R$-module by Proposition \ref{4.1}. Moreover, $(X, e\Lambda)$ is a support $\tau$-tilting
pair in $\mod \Lambda$. Let $(Y, e'\Gamma)$ be a support $\tau$-tilting pair in $\mod \Gamma$. Then
$\Hom_R({{(e\Lambda,0)}\oplus{(e'\Gamma\otimes_\Gamma M, e'\Gamma)}},{(X, 0)}\oplus{(Y\otimes_\Gamma M, Y)})=0$.
Moreover, we have
$$|{(X, 0)}|+|{(Y\otimes_\Gamma M, Y)}|+|{(e\Lambda, 0)}|+|{(e'\Gamma\otimes_\Gamma M, e'\Gamma})|
=|X|+|Y|+|e\Lambda|+|e'\Gamma|=|\Lambda|+|\Gamma|=|R|$$
by Lemma \ref{4.2}. Thus $({(X, 0)}\oplus{(Y\otimes_\Gamma M, Y)}, {(e\Lambda, 0)}\oplus
{(e'\Gamma\otimes_\Gamma M, e'\Gamma)})$ is a support $\tau$-tilting pair in $\mod R$.
\end{proof}

We give another proof of Theorem \ref{4.3} as follows.
\begin{proof}
By Theorem \ref{3.4}, we have that ${(X, 0)}\oplus{(Y\otimes_\Gamma M, Y)}$
is a silting $R$-module if and only if $X$ is a silting $\Lambda$-module with respect to
a projective presentation $\sigma_X$ of $X$, $Y$ is a silting
$\Gamma$-module and $Y\otimes_\Gamma M\in \Gen X(=D_{\sigma_X})$. By \cite[Theorem 4.9]{AMV},
$\sigma_X$ is a 2-term silting complex. Let $e$ be the maximal idempotent such that $\Hom_\Lambda(e\Lambda, X)=0$.
Then by \cite[Theorem 3.2]{AIR}, we have $\sigma_X=\sigma\oplus\sigma'$ with $\sigma$ a minimal projective
presentation of $X$ and $\sigma'$ the complex $e\Lambda\rightarrow 0$.
So the condition $Y\otimes_\Gamma M\in D_{\sigma_X}$ is equivalent to that $Y\otimes_\Gamma M\in D_{\sigma}$ and
$Y\otimes_\Gamma M\in D_{\sigma'}$, and hence is equivalent to that $\Hom_\Lambda(Y\otimes_\Gamma M,\tau X)=0$
and $\Hom_\Lambda(e\Lambda, Y\otimes_\Gamma M)=0$ by Lemma \ref{2.6}. Now Theorem \ref{4.3} follows from Proposition 2.7.
\end{proof}

\begin{remark}\label{4.4}
{\rm \begin{enumerate}
\item[]
\item[(1)] In fact, the maximal idempotent $e$ in Theorem \ref{4.3}(4) is exactly the idempotent
such that $(X, e\Lambda)$ is a support $\tau$-tilting pair in $\mod\Lambda$ (\cite[Proposition 2.3(a)]{AIR}).
\item[(2)] Since $\tau$-tilting modules are exactly sincere support $\tau$-tilting modules by \cite[Proposition 2.2(a)]{AIR},
it follows from Theorem \ref{4.3} that $(X,0)$$\oplus$$(Y\otimes_\Gamma M, Y)$ is a $\tau$-tilting
$R$-module if and only if $X$ is a $\tau$-tilting $\Lambda$-module, $Y$ is a $\tau$-tilting $\Gamma$-module
and $\Hom_\Lambda(Y\otimes_\Gamma M,\tau X)=0$.
\item[(3)] Note that tilting modules are exactly $\tau$-tilting modules whose projective dimension is at most one.
Then by (2), we have that if $_\Gamma M$ is projective, then $(X,0)$$\oplus$$(Y\otimes_\Gamma M,Y)$
is a tilting $R$-module if and only if $X$ is a tilting $\Lambda$-module, $Y$ is a tilting $\Gamma$-module and
$\Hom_\Lambda(Y\otimes_\Gamma M,\tau X)=0$. This result can be induced directly from \cite[Theorem A]{AN}.
Take $S=\Lambda\times\Gamma$ as in Remark \ref{3.5}. Then $(X\oplus Y)\otimes R$ is a tilting $R$-module if and only if $X\oplus Y$
is a tilting $S$-module, $\Hom_R((X\oplus Y)\otimes_SM,\tau(X\oplus Y))=0$
(that is, $\Hom_\Lambda(Y\otimes_\Gamma M,\tau X)=0$) and $\Hom_k({_SM},k)\in \Gen(X\oplus Y)$. Since $\Hom_k({_SM},k)$ is an
injective $S$-module when $_\Gamma M$ is projective, the assertion now is obvious.
\end{enumerate}}
\end{remark}

\begin{corollary}\label{4.5}
For any $Y\in\mod\Gamma$, $(Y\otimes_\Gamma M, Y)$ is a support $\tau$-tilting $R$-module
if and only if $Y$ is a support $\tau$-tilting $\Gamma$-module and $Y\otimes_\Gamma M=0$.
\end{corollary}

Putting $\Gamma=M=\Lambda$ in Theorem \ref{4.3}, we get the following result.

\begin{corollary}\label{4.6}
Let $X,Y\in\mod\Lambda$ and $R=\left(\begin{array}{cc}\Lambda & 0 \\
\Lambda & \Lambda\end{array}\right)$. Then the following statements are equivalent.
\begin{enumerate}
\item[(1)] $(X, 0)$$\oplus$$(Y,Y)$ is a support $\tau$-tilting $R$-module.
\item[(2)] Both $X$ and $Y$ are  support $\tau$-tilting $\Lambda$-modules and
$\Hom_\Lambda(Y,\tau X)=0=\Hom_{\Lambda}(e\Lambda,Y)$, where $e$ is the maximal idempotent
such that $\Hom_\Lambda(e\Lambda, X)=0$.
\end{enumerate}
In particular, $(X, 0)$$\oplus$$(Y, Y)$
is a tilting $R$-module if and only if both $X$ and $Y$ are  tilting $\Lambda$-modules and
$\Hom_\Lambda(Y,\tau X)=0$.
\end{corollary}

Recall from \cite[Chapter XV, Definition 1.1(a)]{SS2007} that the {\it one-point extension} of $\Lambda$ by
the module $M_{\Lambda}$ is the special triangular algebra
$\left(\begin{array}{cc}\Lambda & 0 \\
_kM_{\Lambda} & k\end{array}\right)$. There are only two support $\tau$-tilting $k$-modules: 0 and $k$.
Let $e_{a}$ be the idempotent corresponding to the extension point $a$. Then we have $k\otimes_kM\cong M_\Lambda$ and
${(k\otimes_kM, k)}\cong {(M, k)}=e_aR$. As a consequence of Theorem \ref{4.3}, we get the following

\begin{corollary}\label{4.7}
Let $X\in\mod \Lambda$ and $R=\left(\begin{array}{cc}\Lambda & 0 \\
_kM_{\Lambda} & k\end{array}\right)$. Then we have
\begin{enumerate}
\item[(1)] $X_R$ is support $\tau$-tilting if and only if $X_\Lambda$ is support $\tau$-tilting.
\item[(2)] $X_R \oplus e_{a}R$ is support $\tau$-tilting $R$-module
if and only if $(X, e\Lambda)$ is a support $\tau$-tilting pair in $\mod \Lambda$,
$\Hom_\Lambda(M,\tau X)=0=\Hom_\Lambda(e\Lambda, M)$ for some idempotent $e$ of $\Lambda$.
\item[(3)] $X_R \oplus e_{a}R$ is a tilting $R$-module
if and only if $X$ is a tilting $\Lambda$-module and $\Hom_\Lambda(M,\tau X)=0$.
\end{enumerate}
\end{corollary}

Let $\Lambda$ be an algebra and $_\Lambda M_\Lambda$ a ($\Lambda,\Lambda)$-bimodule. Recall that
$$T(\Lambda, M):=\Lambda\oplus{_{\Lambda}M}\oplus{_{\Lambda}M^2}\oplus\cdots\oplus{_{\Lambda}M^n}\oplus\cdots$$
as an abelian group is called the {\it tensor algebra} of $M$ over $\Lambda$,
where $M^n$ is the $n$-fold $\Lambda$-tensor product $M\oplus{_{\Lambda}M}\oplus\cdots\oplus{_{\Lambda}M}$.
We will assume that $T(\Lambda, M)$ is finite dimensional (equivalently, $M$ is nilpotent).
Note that the triangular matrix algebra $\left(\begin{array}{cc}\Lambda & 0 \\
M & \Gamma\end{array}\right)$ can be viewed as the tensor algebra
$T(\Lambda\times\Gamma, {_{\Lambda\times\Gamma}M_{\Lambda\times\Gamma}})$.
The following result is a generalization of Theorem \ref{4.3}.

\begin{theorem}\label{4.8}
Let $\Lambda$ be an algebra, $_\Lambda M_\Lambda$ a $(\Lambda,\Lambda)$-bimodule
and $X\in\mod\Lambda$, and let
$T(\Lambda,M)$ be the tensor algebra of $M$ over $\Lambda$ and $e$ an idempotent of $\Lambda$.
Then the following statements are equivalent.
\begin{enumerate}
\item[(1)] $(X\otimes_{\Lambda}T(\Lambda,M), eT(\Lambda,M))$ is a support $\tau$-tilting pair
in $\mod T(\Lambda,M)$.
\item[(2)] $(X, e\Lambda)$ is a support $\tau$-tilting pair in $\mod\Lambda$ and
$$\Hom_{\Lambda}(X\otimes{_{\Lambda}N}, \tau X)=0=\Hom_{\Lambda}(e\Lambda, {X\otimes{_{\Lambda}N}}),$$
where $N=M\oplus{_{\Lambda}M^2}\oplus\cdots\oplus{_{\Lambda}M^n}\oplus\cdots$.
\end{enumerate}
\end{theorem}

\begin{proof}
Assume that $\sigma_X$ is a minimal projective presentation of $X$.
Then $\sigma_X\otimes T(\Lambda,M)$ is a minimal projective presentation of $X\otimes_{\Lambda}T(\Lambda,M)$.
Write $\sigma:=\sigma_X\otimes T(\Lambda,M)\oplus (eT(\Lambda,M)\to 0)$.
Note that there exists a natural projection from $T(\Lambda,M)$ to $\Lambda$.
It follows from \cite[Theorems 3.2]{AIR} and \cite[Theorem 2.2]{B}
that $(X\otimes_{\Lambda}T(\Lambda,M), eT(\Lambda,M))$ is a support $\tau$-tilting pair in $\mod T(\Lambda,M)$
if and only if $(\sigma_X\oplus (e\Lambda \to 0))\otimes T(\Lambda,M)$ is a 2-term silting complex,
and if and only if $\sigma_X\oplus (e\Lambda \to 0)$ is a 2-term silting complex and
$\Hom_{T(\Lambda,M)}(T(\Lambda,M)_{\Lambda}, X\otimes_{\Lambda}T(\Lambda,M))\in \Gen X=D_{\sigma_X\oplus (e\Lambda \to 0)}$.
Since
$$\Hom_{T(\Lambda,M)}(T(\Lambda,M)_{\Lambda}, X\otimes T(\Lambda,M))\cong
X\otimes T(\Lambda,M)_{\Lambda}\cong X\oplus X\otimes N,$$
we have that $\Hom_{T(\Lambda,M)}(T(\Lambda,M)_{\Lambda}, X\otimes_{\Lambda}T(\Lambda,M))\in \Gen X$ if and only if
$X\otimes_{\Lambda}N\in \Gen X$, that is, $X\otimes_\Lambda N\in  D_{\sigma_X}$ and
$X\otimes_\Lambda N\in D_{(e\Lambda \to 0)}$. Now the assertion follows from Lemma \ref{2.6}.
\end{proof}

We recall some notions from \cite{FL2012}.
A {\it pseudovalued graph} $(\mathcal{G},\mathcal{D})$ consists of
\begin{enumerate}
\item[(i)] a finite set $\mathcal{G} = \{1, 2, \cdots,n \}$ whose elements are called {\it vertices}; and
\item[(ii)] a correspondence taking any ordered pair $(i, j)\in \mathcal{G}\times \mathcal{G}$ to a non-negative
integer $d_{ij}$ such that if $d_{ij}\neq 0$ then $d_{ji}\neq 0$.
If $d_{ij}\neq 0$, then such a pair $(i, j)$ is called an {\it edge} between the vertices $i$ and $j$.
\end{enumerate}
The family $\mathcal{D}=\{(d_{ij},d_{ji})\mid(i, j)\in \mathcal{G}\times \mathcal{G}\}$ is called a
{\it valuation} of the graph $\mathcal{G}$. A {\it pseudovalued~quiver} is a pseudovalued graph
$(\mathcal{G},\mathcal{D})$ with an {\it orientation} which is given by prescribing for each edge
an ordering, indicated by an oriented edge. A {\it path} from $j$ to $i$ of the pseudovalued
quiver $(\mathcal{G},\mathcal{D})$ is a sequence $k_1(=j),k_2,\cdots, k_t(=i)$ of vertices
such that there is a valued oriented edge from $k_s$ to $k_{s+1} $ for any $s=1,2\cdots, t-1$.

For an algebra $\Lambda$ and a finitely generated left (resp. right) $\Lambda$-module $M$,
the {\it rank} $\rank_{\Lambda}M$ (resp. $\rank M_{\Lambda}$) of $M$ is defined as the minimal
cardinal number of the sets generators of $M$ as a left (resp. right) $\Lambda$-module.
A {\it $k$-pseudomodulation} $\mathcal{M}= (\Lambda_i,{_iM_j})$ of a pseudovalued graph
$(\mathcal{G},\mathcal{D})$ is defined as a set of $k$-algebras $\{\Lambda_i\}_{i\in \mathcal{G}}$,
together with a set $\{_iM_j \}_{(i, j )\in \mathcal{G}\times \mathcal{G}} $ of finitely generated
$(\Lambda_i,\Lambda_j)$-bimodules $_iM_j$ such that $\rank(_iM_j)_{\Lambda_j} =d_{ij}$ and
$\rank_{\Lambda_i}(_iM_j)=d_{ji}$. 

As a consequence of Theorem \ref{4.8}, we have the following result.

\begin{corollary}\label{4.9}
Let $\mathcal{M}=(\Lambda_i,{_iM_j})$ be a $k$-pseudomodulation of a pseudovalued quiver
$(\mathcal{G},\mathcal{D})$ and $V_i\in \mod \Lambda_i$ for any $i\in \mathcal{G}$,
and let $$\Lambda=\bigoplus_{i\in \mathcal{G}}\Lambda_i,\ \
M=\bigoplus_{(i,j)\in \mathcal{G}\times\mathcal{G}}\;{_iM_j},$$
$$X_i=\bigoplus_{j\in Q_i} V_j\otimes{_jM_{k_2}}\otimes\cdots\otimes{_{k_{s-1}}M_i},$$
where $Q_i=\{j\in \mathcal{G}\mid$ there is a path from $j$ to $i\}$
and $k_1(=j), k_2,\cdots, k_s(=i)$ is a path from $j$ to $i$.
Then the following statements are equivalent.
\begin{enumerate}
\item[(1)] $\bigoplus_{i\in\mathcal{G}}(V_i\oplus X_i)$ is a support $\tau$-tilting
$T(\Lambda,M)$-module.
\item[(2)] $\bigoplus_{i\in\mathcal{G}}V_i$ is a support $\tau$-tilting $\Lambda$-module and
$$\Hom_{\Lambda}(\bigoplus_{i\in\mathcal{G}}X_i,\tau (\bigoplus_{i\in\mathcal{G}}V_i))
=0=\Hom_{\Lambda}(e\Lambda,\bigoplus_{i\in \mathcal{G}}X_i),$$
where $e$ is an idempotent of $\Lambda$ such that $(\bigoplus_{i\in\mathcal{G}}V_i, e\Lambda)$
is a support $\tau$-tilting pair in $\mod \Lambda$.
\item[(3)] For any $i\in\mathcal{G}$, $V_i$ is a support $\tau$-tilting $\Lambda_i$-module and
$$\Hom_{\Lambda_i}(X_i,\tau V_i)=0=\Hom_{\Lambda_i}(e_i\Lambda_i, X_i),$$
where $e_i$ is an idempotent of $\Lambda_i$ such that $(V_i, e_i\Lambda_i)$ is a support
$\tau$-tilting pair in $\mod \Lambda_i$.
\end{enumerate}
\end{corollary}

\begin{proof}
The assertion $(2)\Leftrightarrow (3)$ is clear.

Set $X:=\bigoplus _{i\in \mathcal{G}}V_i$ and
$N:=M\oplus{_{\Lambda} M^2}\oplus\cdots\oplus{_{\Lambda}M^n}\oplus\cdots$. Then
we have $X\otimes _{\Lambda}N\cong \bigoplus_{i\in \mathcal{G}}{X_i}$, and hence
$$X\otimes _{\Lambda}T(\Lambda,M)\cong X\oplus (X\otimes _{\Lambda}N)\cong X\oplus
(\bigoplus_{i\in \mathcal{G}}{X_i})\cong \bigoplus_{i\in  \mathcal{G}}(V_i\oplus X_i).$$
Now the assertion $(1)\Leftrightarrow (2)$ follows from Theorem \ref{4.8}.
\end{proof}

\begin{remark}\label{4.10}
{\rm Let $\mathcal{M}$ be a $k$-pseudomodulation of a pseudovalued quiver $(\mathcal{G},\mathcal{D})$.
Recall from \cite{FL2012} that a {\it representation} of $\mathcal{M}$
is an object $(V_i,{_j\varphi_i})$, where to each vertex $i\in \mathcal{G}$ corresponds a $\Lambda_i$-module
$V_i$ and to each oriented edge $i\to j$ corresponds a $\Lambda_j$-homomorphism
$_j\varphi_ i: V_i \otimes_{\Lambda_i}{_iM_j}\to V_j$. If each $V_i$ is finitely generated as
$\Lambda_i$-module, then the representation $(V_i,{_j\varphi_i})$ is called {\it finitely generated}.
We use $\rep(\mathcal{M})$ to denote the category consisting of all finitely generated representations
of $\mathcal{M}$. It was shown in \cite[Theorem 3.2]{FL2012} that $\rep(\mathcal{M})$ is equivalent to
$\mod T(\Lambda,M)$, where $\Lambda$ and $M$ are as in Corollary \ref{4.9}.}
\end{remark}

\section{An example}

In this section, for a finite dimensional $k$-algebra over an algebraically closed field $k$ with the quiver $Q$,
we use $P_i$ (resp. $S_i$) to denote the indecomposable projective (resp. simple) module corresponding
to the vertex $i$ in $Q$, and use $e_i$ to denote the idempotent corresponding
to the vertex $i$. For brevity, the symbol $\oplus$ between modules will be omitted; for example, for modules $M$ and $N$,
we will replace $M\oplus N$ with $MN$. We illustrate some of our work with the following example.

\begin{example}
{\rm Let $R$ be a finite dimensional $k$-algebra over $k$ given by the following quiver
$$
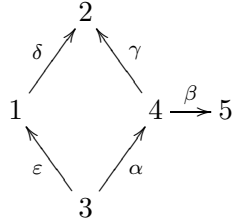
\xymatrix@C=15pt{&2&&\\
1\ar[ru]^{\delta}&&4\ar[lu]_{\gamma}\ar[r]^{\beta}&5\\
&3\ar[lu]^{\varepsilon}\ar[ru]_{\alpha}&&}$$
with the relation $\alpha\gamma=\varepsilon\delta$ and $\alpha\beta=0$. Let $e=e_1+e_2$. Then
$$R=\left(\begin{array}{cc}eRe & 0 \\
(1-e)Re & (1-e)R(1-e)\end{array}\right).$$
Take $$\Lambda:=eRe(\cong k(\xymatrix@C=15pt{1\ar[r]^{\delta}&2})),$$
$$\Gamma:=(1-e)R(1-e)(\cong k(\xymatrix@C=15pt{3\ar[r]^{\alpha}&4\ar[r]^{\beta}&5})\ \text{with}\
\alpha\beta=0),$$
$${_\Gamma M_\Lambda=(1-e)Re}.$$
Then $M_\Lambda\cong \Lambda$ and $_\Gamma M\cong P_3S_4$.

For an indecomposable $\Lambda$-module $X$ and an indecomposable $\Gamma$-module $Y$,
we use $(X, Y)$ to denote the corresponding indecomposable $R$-module.
The  Auslander-Reiten quiver of $R$ is as follows.
$$
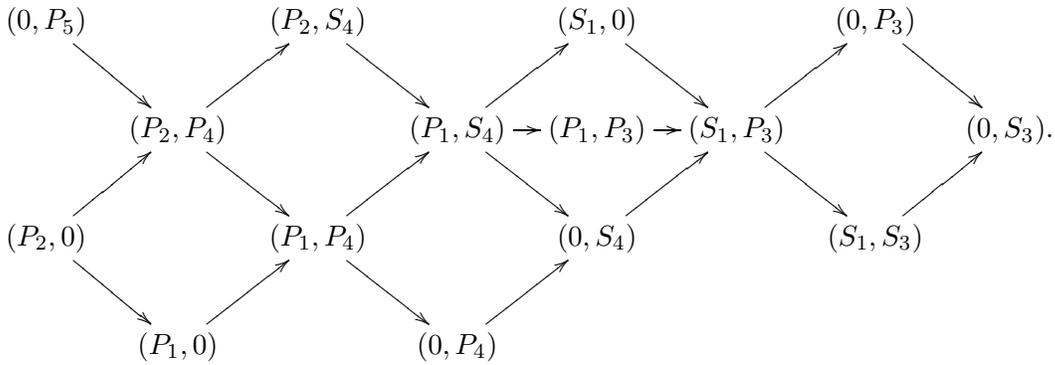
\xymatrix@C=10pt{(0, P_5)\ar[rd]&&(P_2,S_4)\ar[rd]&&(S_1, 0)\ar[rd]&&(0,P_3)\ar[rd]&\\
&(P_2, P_4)\ar[ru]\ar[rd]&&(P_1, S_4)\ar[ru]\ar[rd]\ar[r]&(P_1, P_3)\ar[r]&(S_1, P_3)\ar[ru]\ar[rd]&&{(0, S_3)}.\\
(P_2, 0)\ar[ru]\ar[rd]&&(P_1, P_4)\ar[ru]\ar[rd]&&(0,S_4)\ar[ru]&&{(S_1, S_3)}\ar[ru]&\\
&(P_1, 0)\ar[ru]&&(0, P_4)\ar[ru]&&&&}$$

The Hasse quivers of $\Lambda$ and $\Gamma$ are as follows.
$$
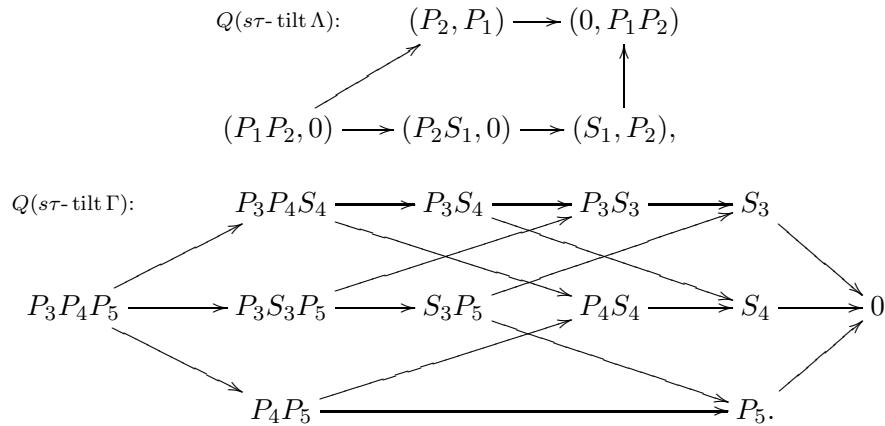
\xymatrix@C=15pt{{\smallmatrix Q(s\tau{\text -}\tilt \Lambda):
\endsmallmatrix} &(P_2,P_1)\ar[r]&(0, P_1P_2)\\
 (P_1P_2,0)\ar[r]\ar[ru]&(P_2S_1,0)\ar[r]&(S_1,P_2),\ar[u]
}$$
\begin{center}
$\xymatrix@C=30pt{{\smallmatrix Q(s\tau{\text -}\tilt \Gamma):
\endsmallmatrix}&P_3P_4S_4\ar[r]\ar[rrd]&P_3S_4\ar[r]\ar[rrd]&P_3S_3\ar[r]&S_3\ar[rd]&\\
P_3P_4P_5\ar[r]\ar[ru]\ar[rd]&P_3S_3P_5\ar[r]\ar[rru]&S_3P_5\ar[rru]\ar[rrd]&P_4S_4\ar[r]&S_4\ar[r]&0\\
&P_4P_5\ar[rrr]\ar[rru]&&&P_5.\ar[ru]&
}$
\end{center}
Now we list $T\otimes_\Gamma M$ for all support $\tau$-tilting $\Gamma$-modules $T_\Gamma$ in the following table.
\begin{center}
\begin{tabularx}{\textwidth}{|X|c|c|c|c|c|c|c|c|c|c|c|c|}
\hline
$T_i$&$P_3P_4P_5$&  $P_3P_4S_4$  &$P_3S_3P_5$&  $P_4P_5$ &$P_3S_4$&$S_3P_5$ &$P_3S_3$&  $P_4S_4$  &$S_3$&$S_4$&$P_5$&0\\ \hline
$T_i\otimes M$&$P_1P_2$&  $P_1P_2P_2$  &$P_1S_1$&  $P_2$&$P_1P_2$&$S_1$&$P_1S_1$&  $P_2P_2$ &$S_1$&$P_2$&0&0\\ \hline
\end{tabularx}
\end{center}

By Remark \ref{4.4}(1), we have that the maximal idempotent $e$ in Theorem \ref{4.3}
is exactly the idempotent such that $(X, e\Lambda)$ is a support $\tau$-tilting pair in $\mod\Lambda$.
We can construct many support $\tau$-tilting $R$-modules by Theorem \ref{4.3}.

(1) Considering the support $\tau$-tilting pair $(P_1P_2,0)$ over $\Lambda$, we have $\tau (P_1P_2)=0$ and $e\Lambda=0$.
Hence, we get the following support $\tau$-tilting $R$-modules.
\begin{center}
$(P_1,0)$$(P_2,0)$$(P_1, P_3)$$(P_2, P_4)$$(0, P_5)$, $(P_1,0)$$(P_2,0)$$(P_1,P_3)$$(P_2,P_4)$$(P_2, S_4)$,
$(P_1,0)$$(P_2,0)$$(P_1, P_3)$$(S_1, S_3)$$(0, P_5)$, $(P_1,0)$$(P_2,0)$$(P_2, P_4)$$(0,P_5)$,\\
$(P_1,0)$$(P_2 ,0)$$(P_1, P_3)$$(P_2, S_4)$, $(P_1,0)$$(P_2,0)$$(S_1, S_3)$$(0, P_5)$,\\
$(P_1,0)$$(P_2,0)$$(P_1, P_3)$$(S_1, S_3)$,$(P_1,0)$$(P_2,0)$$(P_2, P_40$$(P_2, S_4)$,\\
$(P_1,0)$$(P_2,0)$$(S_1, S_3)$, $(P_1,0)$$(P_2,0)$$(P_2, S_4)$, \\
$(P_1,0)$$(P_2,0)$$(0, P_5)$, $(P_1,0)$$(P_2,0)$.
\end{center}

(2) Considering the support $\tau$-tilting pair $(P_1S_1, 0)$ over $\Lambda$, we have $\tau (P_1S_1)=P_2$ and $e\Lambda=0$.
Hence, those $R$-modules
\begin{center}
$(P_1,0)$$(S_1,0)$$(P_1, P_3)$$(S_1, S_3)$$(0, P_5)$, $(P_1,0)$$(S_1,0)$$(S_1,S_3)$$(0, P_5)$, \\
$(P_1,0)$$(S_1,0)$$(P_1,P_3)$$(S_1, S_3)$, $(P_1,0)$$(S_1,0)$$(S_1,S_3)$,\\
$(P_1,0)$$(S_1,0)$$(0, P_5)$, $(P_1,0)$$(S_1,0)$
\end{center}
are support $\tau$-tilting.

(3) Considering the support $\tau$-tilting pair $(P_2, P_1)$ over $\Lambda$, we have $\tau P_2=0$ and $e\Lambda=P_1$. Hence,
we get the following support $\tau$-tilting $R$-modules.
\begin{center}
$(P_2,0)$$(P_2,P_4)$$(0, P_5)$, $(P_2,0)$$(P_2,P_4)$$(P_2,S_4)$,
$(P_2,0)$$(P_2, S_4)$, $(P_2,0)$$(0, P_5)$, $(P_2,0)$.
\end{center}

(4) Considering the support $\tau$-tilting pair $(S_1, P_2)$ over $\Lambda$, we have $\tau S_1=P_2$ and $e\Lambda=P_2$.
Hence, we get the following support $\tau$-tilting $R$-modules.
\begin{center}
$(S_1,0)$$(S_1, S_3)$$(0, P_5)$, $(S_1,0)$$(S_1, S_3)$, $(S_1,0)$$(0, P_5)$, $(S_1,0)$.
\end{center}

(5) Considering the support $\tau$-tilting pair $(0, P_1P_2)$ over $\Lambda$, we have  $e\Lambda=P_1P_2$.
Hence, we get  support $\tau$-tilting $R$-modules: $(0,P_5)$ and $0$.

(6) Considering the tilting $\Gamma$-module $P_3P_4S_4$ which is a silting module with respect to
$$\sigma:P_5\rightarrow P_3P_4P_4,$$
we have that $\sigma\otimes_\Gamma M:0\rightarrow P_1P_2P_2$ is monic.
Thus $(P_1,0)$$(P_2,0)$$(P_1, P_3)$$(P_2, P_4)$$(P_2, S_4)$ is a tilting $R$-module by Corollary \ref{3.7},
though $_\Gamma M\cong P_3S_4$ is not flat.

(7) Unfortunately, we can not get all support $\tau$-tilting $R$-modules by Theorem \ref{4.3}.
For example, the module $(0, P_3)$$(P_2,0)$$(P_1, P_3)$$(P_2, P_4)$$(0, P_5)$ is a support $\tau$-tilting $R$-module,
but it does not appear in (1)--(5).}
\end{example}

\end{document}